\documentclass[12pt,oneside]{amsart}

\usepackage{amscd}
\usepackage{amssymb}
\usepackage{amstext}
\usepackage{amsthm}
\usepackage{mathrsfs}

\usepackage{latexsym}
\usepackage{xcolor}
\numberwithin{equation}{section}

\DeclareMathOperator{\dist}{dist}
\DeclareMathOperator{\Span}{Span}
\DeclareMathOperator{\supp}{supp}
\DeclareMathOperator{\card}{card}

\renewcommand{\phi}{\varphi}

\newtheorem{Thm}{Theorem}
\newtheorem{theorem}[Thm]{Theorem}
\newtheorem{lemma}[Thm]{Lemma}

\newtheorem{proposition}[Thm]{Proposition}

\newtheorem{definition}[Thm]{Definition}

\begin{document}
\sloppy
\title[On the chain structure in the de Branges spaces]
{On the chain structure in the de Branges spaces}

\author{Yurii Belov, Alexander Borichev}
\address{Yurii Belov,
\newline Department of Mathematics and Computer Sciences, St.~Petersburg State University, St. Petersburg, Russia 
\newline {\tt j\_b\_juri\_belov@mail.ru}
\smallskip
\newline \phantom{x}\,\, Alexander Borichev,
\newline Aix-Marseille University, CNRS, Centrale Marseille, I2M, Marseille, France
\newline {\tt alexander.borichev@math.cnrs.fr}
}

\keywords{Hilbert spaces of entire functions, de Branges spaces, indivisible intervals, exponential type.}

\subjclass[2020]{Primary: 46E22, Secondary: 30D15, 47B32}

\thanks{
The first author was supported by a grant of the Government of the Russian Federation for the state support of scientific research, carried out under 
the supervision of leading scientists, agreement 075-15-2021-602 and Russian Foundation for Basic Research grant 20-51-14001-ANF-a.
The second author was supported by the project ANR 18-CE40-0035.} 

\begin{abstract} We study the indivisible intervals and the monotonicity of the growth of the exponential type in the chains of de Branges subspaces 
in terms of the spectral measure. We prove that for spectral measures supported on $\mathbb Z$, there exist at most two subspaces of the same type, which then bound an indivisible interval. Furthermore, in this case, we study possible locations of the indivisible intervals.
\end{abstract}

\maketitle

\section{Introduction and main results}

One of the main parts of the de Branges theory of Hilbert spaces of entire functions is the study of the chains of the de Branges subspaces 
of a given de Branges space. There are different equivalent definitions of de Branges spaces (axiomatic approach, in terms of 
Hermite--Biehler functions, in terms of the Hamiltonians, as the weighted Cauchy transforms, spaces isometrically imbedded into $L^2$ with respect to a measure on the real line). For more information on the de Branges theory see \cite{bo1,bo2,bo3} and the references therein. For some recent progress see, for example, \cite{re1,re3}. 

In this paper we use the weighted Cauchy transform definition of the de Branges spaces. 
Let $T=\{t_n\}_{n\in\mathcal{N}}$ be a discrete subset of the real line and let $\mu=\sum_{n\in\mathcal{N}}{\mu_n\delta_{t_n}}$ be a positive measure such that 
\begin{equation}
\sum_{n\in\mathcal N}\frac{\mu_n}{t^2_n+1}<\infty.
\label{mu1}
\end{equation}
Fix an entire function $A$ real on the real line with simple zeros at $T=\supp\mu$ and define the corresponding de Branges space (in the weighted Cauchy transform form)
$\mathcal{HC}(\mu)$ as follows:
$$
\mathcal{HC}(\mu)=\mathcal{HC}(A,\mu)=\biggl\{f(z)=A(z)\sum_{n\in\mathcal N}\frac{a_n\mu^{1\slash2}_n}{z-t_n}: \{a_n\}_{n\in\mathcal N}\in \ell^2\biggr\},
$$
$$
\|f\|_{\mathcal{HC}(\mu)}=\|a_n\|_{\ell^2}.
$$

We study an important class of de Branges spaces corresponding to the so called canonical systems
on finite interval. Namely, given a $2\times2$ real summable a.e. positively semi-defined matrix function $H$ (Hamiltonian) on a finite interval $[0,L]$, we consider the  system 
$$
J Y'(t)=zH(t)Y(t),\qquad t\in[0,L],\qquad J=\biggl{(}\begin{array}{ccc} 0 & -1 \\
1 & 0
\end{array}\biggr{)},
$$
 where $z\in\mathbb C$ is the so called spectral parameter, $Y$ is an absolutely continuous column vector function such that $Y(0)^T=(0,1)$ and
 $Y(L)^T=(A,B)$. The entire functions $A$ and $B$ are real on the real line with simple real interlacing zeros. We define a measure $\mu$ supported on $\mathcal{Z}_A$, 
where $\mathcal{Z}_F$ is the zero set of an entire function $F$, with masses $B(t)/A'(t)$, $t\in \mathcal{Z}_A$, and associate to the canonical system with Hamiltonian $H$ the de Branges space
$\mathcal{HC}(A,\mu)$.

It is known that a de Branges space $\mathcal{HC}(A,\mu)$ corresponds to a canonical system
on a finite interval if and only if it is regular, that is,
$$
z\mapsto \frac{f(z)-f(w)}{z-w}\in\mathcal{HC}(A,\mu),
$$
whenever $f\in\mathcal{HC}(A,\mu)$, $w\in\mathbb C$.
This is equivalent to the condition that the functions in our space are in the Cartwright class and  
\begin{equation}
\sum_{n\in\mathcal N}\frac{1}{\mu_n(t^2_n+1)A'(t_n)^2}<\infty,
\label{mu2}
\end{equation}
and from now on we assume that this condition is satisfied.

By the de Branges theory, the de Branges subspaces of the space $\mathcal{HC}(\mu)$ constitute
a chain $\{\mathcal{HC}(\mu_s)\}$ ordered by inclusion. Every such $\mathcal{HC}(\mu_s)$ corresponds to a point $s$ or to a subinterval 
$(a,s]$ of $(0,L]$ and to the restriction of our Hamiltonian $H$ to $[0,s]$. We set $(A_s,B_s)=Y(s)^T$ and 
we associate to the space $\mathcal{HC}(\mu_s)$ the corresponding entire function $A_s$. 
We have $\mathcal{Z}_{A_s}=\supp \mu_s$. 
Given the de Branges space $\mathcal{HC}(\mu)$ we denote the corresponding chain of the de Branges subspaces 
by ${\rm Chain}(\mu)$.
We are interested in the so called $H$-indivisible intervals in this chain,
that is the semi open maximal intervals $I=(a,s]\subset[0,L]$ such that $H$ is a degenerate constant matrix on $I$ and, hence, the 
subspaces $\mathcal{HC}(\mu_s)$ coincide for $s\in I$.

If the spectrum of $\mu$ is $\mathbb Z$, then one can easily verify (see Section~\ref{se2} below) that 
the chain does not contain indivisible end intervals $(0,a]$ and $(a,L]$. 
It is of interest to study the indivisible intervals inside the chain.

\begin{theorem} If $\supp(\mu)=\mathbb Z$, and the space $\mathcal{HC}(\mu)$ is regular, then the corresponding chain ${\rm Chain}(\mu)$ can contain one indivisible interval 
and cannot contain two contiguous indivisible intervals. 
\label{onetwoint}
\end{theorem}

When the support of the measure $\mu$ is $\mathbb Z$, such chains may contain infinitely many non-contiguous indivisible intervals, see Theorem~\ref{th2} below. 
On the other hand, these indivisible intervals are somewhat separated, see Theorem~\ref{thm6}. 

Furthermore, if the support of the measure $\mu$ is similar enough to $\mathbb Z$, then we get results analogous to Theorem~\ref{onetwoint}, see Section 3.

\subsection{Exponential type}

Another important characteristic of a de Branges space $\mathcal{HC}(\mu)$ is the exponential type ${\rm Type}(\mathcal{HC}(\mu))$, that is the exponential type $t(A)$ of the function $A$. 

\begin{theorem} Let $\Sigma$ be a countable subset of the interval $(0,\pi)$. There exists a measure 
$\mu$ supported on $\mathbb Z$ such that the space $\mathcal{HC}(\mu)$ is regular, and 
the corresponding chain ${\rm Chain}(\mu)$ contains indivisible intervals $J_s$ with 
${\rm Type}(\mathcal{HC}(\mu_t))=s$, for all $t\in J_s$, $s\in\Sigma$.
\label{th2}
\end{theorem}

One can express the exponential type of a de Branges space $\mathcal{HC}(\mu)$ 
in terms of the Hamiltonian $H$ of the corresponding canonical system. Namely, the  
Krein--de Branges formula (see e.g. \cite[Theorem 11]{bo3}) states that 
$$
{\rm Type}(\mathcal{HC}(\mu))=\int_0^L\sqrt{\det H(t)}\,dt.
$$

Let us recall here some key facts 
from the de Branges theory which are necessary to formulate our results. 

\begin{theorem}[\text{\cite[Theorem 17]{bo3}}] Given a positive measure $\nu$ on $\mathbb R$ such that 
\begin{equation}
\int_\mathbb{R}\frac{d\nu(t)}{1+t^2}<\infty,
\label{q11}
\end{equation}
there exists a chain of regular de Branges spaces $\mathcal H_{t,\nu}$, $(t\in(0,\infty)$ or $t\in(0,L])$ such that $\mathcal{H}_{t,\nu}$ is isometrically embedded in $L^2(\nu)$,
the set $\bigcup_t\mathcal H_{t,\nu}$ is dense in $L^2(\nu)$ and
$$
\mathcal H_{t_1,\nu}\subset\mathcal H_{t_2,\nu},\qquad t_1\le t_2.
$$
Every regular de Branges space isometrically embedded in $L^2(\nu)$ belongs to this chain. 
Furthermore, if $L<\infty$, then $T=\supp \nu$ is discrete, and for an entire function $A$ real on the real line with simple zeros at $T$, we have $\mathcal H_{L,\nu}=\mathcal{HC}(A,\nu^*)$, where $\nu^*(\{t\})=1/(\nu(\{t\})(A'(t))^2)$, $t\in\supp T$. Furthermore, $\mathcal H_{L,\nu}$ restricted to $T$ is equal to $L^2(\nu)$.
\label{ftr}
\end{theorem}

For example if $\nu$ is the Lebesque measure, then the corresponding de Branges chain consists of the Paley--Wiener spaces 
$\mathcal{PW}_a$, $a\in(0,\infty)$.


\medskip
The number $T=\sup_t({\rm Type}(\mathcal H_{t,\nu}))$ is called {\it the exponential type} of the measure $\nu$. One of the fundamental question of harmonic analysis
is to determine $T$ via $\nu$, see \cite{BS}, \cite{Po} and the references therein. We are interested in the closely related question about the {\it regularity} of growth of the exponential type.

\begin{definition} Given a positive measure $\nu$ on the real line satisfying \eqref{q11}, we say 
that $\nu$ generates a thin chain if for any type $t>0$ there exists at most one element $\mathcal H_{t,\nu}$
of the chain such that ${\rm Type}(\mathcal H_{t,\nu})=t$.
\end{definition}

In particular, Hamiltonians corresponding to thin chains satisfy the condition $\det H\ne 0$ a.e. on any interval and vice versa.

Given a measure $\nu$ satisfying \eqref{q11}, we define the {\it Hilbert transform} of $\nu$ by the formula 
\begin{equation}
\widetilde{\nu}(z)=\frac{1}{\pi}
\int_{\mathbb{R}}\biggl{(}\frac{1}{z-t}+\frac{t}{t^2+1}\biggr{)}\,d\nu(t).
\label{ht}
\end{equation}
The function $\widetilde{\nu}(z)$ is well defined for $z\in\mathbb{C}\setminus\mathbb{R}$. Moreover, it is well-known that for absolutely continuous measures $\nu$, 
$d\nu=w\,dt$, the Hilbert transform $\widetilde{\nu}$ exists on $\mathbb R$ if we understand the right-hand side of \eqref{ht} in the principal value sense. The Hilbert transform naturally appears
in many problems of harmonic analysis.

If the weight $w $ has convergent {\it logarithmic integral}, then it is well-known that the type of $w(x)dx$ is infinite (see \cite{BS}, \cite{Po}).
We are able to show that under some additional regularity assumptions, the measure $w(x)dx$ generates a thin chain.

\begin{theorem} Let $w$ be a $C^1$ smooth positive function such that
$$
w\in L^1\biggl{(}\frac{dx}{1+x^2}\biggr{)},\qquad \int_{\mathbb{R}}\frac{\log w(x)}{1+x^2}\,dx>-\infty,\qquad (\widetilde{\log w})'\in L^\infty(\mathbb{R}).
$$
Then the measure $w(x)dx$ generates a thin chain. 
\label{regth}
\end{theorem}

Sometimes the chain generated by a measure is not thin, but ``almost thin", that is, for every $t>0$ the chain contains at most two subspaces of type $t$. In particular, this is the case if the support of the measure is $\mathbb Z$, as shows the following result.

\begin{theorem} Let $\nu$ be a positive measure satisfying \eqref{q11} such that $\supp\nu=\mathbb Z$ 
and the corresponding chain of regular de Branges spaces $\mathcal H_{t,\nu}$ is defined on a finite interval $(0,L]$. 
Then for any $0<t_2<t_1\le L$ such that 
${\rm Type}(\mathcal H_{t_1,\nu})={\rm Type}(\mathcal H_{t_2,\nu})$, we have
$$
\dim(\mathcal H_{t_1,\nu}\ominus\mathcal H_{t_2,\nu})\le 1.
$$
\label{thm6}
\end{theorem}

Starting with a measure $\mu=\sum_{n\in\mathbb Z}\mu_n\delta_n$ satisfying the conditions of Theorem~\ref{onetwoint}, we can define 
$\nu=\sum_{n\in\mathbb Z}\mu^{-1}_n\delta_n$ satisfying the conditions of Theorem~\ref{thm6} and such that 
$\mathcal H_{L,\nu}=\mathcal{HC}(\sin\pi z,\mu)$, see Theorem~\ref{ftr}. Thus, Theorem~\ref{thm6} extends the result of Theorem~\ref{onetwoint} on the absence 
of contiguous indivisible intervals.

\subsection{Notation and organization of the paper} 
In this text, $A\lesssim B$ means that $A\le CB$ with a positive constant $C$,
$A \gtrsim B$ means $A\ge cB$ with a constant $c>0$, and $A\asymp B$ means that $A\lesssim B$ and
$A\gtrsim B$ simultaneously.

Some function theoretic criteria for the existence of (contiguous) indivisible intervals in a chain of the Branges spaces are given in Section~\ref{se2}. 
In Section~\ref{se3} we consider the indivisible intervals in the chains associated with the de Branges spaces represented as the Cauchy transforms with spectrum 
in $\mathbb Z$ or some perturbations of $\mathbb Z$. In Section~\ref{se4} we deal with the de Branges subspaces of the same type in a chain. 
Section~\ref{se5} describes some properties of different isometric Cauchy transform representations for the de Branges spaces.

\section{Indivisible intervals}
\label{se2}

In this section we deal with regular de Branges spaces. We start with some equivalent conditions for the existence of an indivisible  
interval in a de Branges chain.

\begin{lemma} 
\label{le7}
Given a de Branges space $\mathcal{HC}(\mu)$, the following assertions are equivalent:
\begin{enumerate}
\item[(i)] The chain ${\rm Chain}(\mu)$ contains an indivisible interval. 
\item[(ii)] For some subspaces in the chain, we have $\dim(\mathcal{HC}(\mu_s)\ominus \mathcal{HC}(\mu_a))=1$.
\item[(iii)] For some subspace $\mathcal{HC}(A,\nu)=\mathcal{HC}(\mu_s)$ in the chain, we have $A\sum_{n\in\mathcal N}\nu_n/(\cdot-t_n)\in \mathcal{HC}(A,\nu)$, where 
$\nu=\sum_{n\in\mathcal N}\nu_n\delta_{t_n}$.
\item[(iv)] There exists a function $G$ in $\mathcal{HC}(\mu)$ real on the real line with simple real zeros
such that $G$ is orthogonal to $G/(\cdot-\lambda)$, $\lambda\in\mathcal{Z}_G$.
\item[(v)] For some subspace $\mathcal{HC}(\mu_s)$ in the chain, the measure $\mu_s$ is finite.
\item[(vi)] For some subspace $\mathcal{HC}(\mu_s)$ in the chain, the domain of the operator of multiplication by $z$ is not dense in $\mathcal{HC}(\mu_s)$.
\end{enumerate}
Under the conditions of {\rm(iv)}, the de Branges space $\mathcal H$ spanned by $G$ and $G/(\cdot-\lambda)$, $\lambda\in\mathcal Z_G$, has exponential type equal to that of $G$. 
\end{lemma}

\begin{proof}
For the implication (i)$\implies$(ii) see \cite[Section 4.3]{bo3}. The implication (ii)$\implies$(i) is evident, see \cite[Problem 86]{bo1}. By 
\cite[Theorem 29]{bo1}, we obtain the implication (ii)$\implies$(iii). Next, taking $G=A\sum_{n\in\mathcal N}\nu_n/(\cdot-t_n)$, and using that $G\perp G/(\cdot-\lambda)$, $\lambda\in\mathcal{Z}_G$, we obtain the implication (iii)$\implies$(iv). The implication (iv)$\implies$(ii) is evident, because the closed space spanned by $G/(\cdot-\lambda)$, $\lambda\in\mathcal{Z}_G$, is a de Branges subspace of $\mathcal{HC}(\mu)$. The equivalence (iii)$\iff$(v) follows because (iii) means that the sequence $\{\nu_n^{1/2}\}_{n\in\mathcal N}$ is in $\ell^2$. The equivalence (iii)$\iff$(vi) follows from \cite[Theorem 29]{bo1}.

\end{proof}

The chain ${\rm Chain}(\mu)$ starts with an indivisible interval $(0,a]$ (or several contiguous indivisible intervals $(0,a_1], (a_1,a_2],\ldots,(a_{k-1}, a_k]$) if and only if $1\in\mathcal{HC}(\mu)$ or, correspondingly, $1,\ldots, z^{k-1}\in\mathcal{HC}(\mu)$) if and only if
$$
\sum_{n\in\mathcal N}\frac{1}{\mu_nA'(t_n)^2}<\infty
$$
or, correspondingly,
$$
\sum_{n\in\mathcal N}\frac{t^{2(k-1)}_n}{\mu_nA'(t_n)^2}<\infty.
$$
The chain ends with $k$ contiguous indivisible intervals $(a_1,a_2], (a_2,a_3],\ldots,(a_{k}, L]$ if and only if
$$
\sum_{n\in\mathcal N}\mu_nt_n^{2(k-1)}<\infty.
$$

Furthermore, the chain ${\rm Chain}(\mu)$ contains $k$ contiguous indivisible intervals if and only if we can find an entire function 
$G$  real on the real line with simple real zeros
such that $G$ is orthogonal to $G\slash(\cdot-\lambda)$, $\lambda\in\mathcal{Z}_G$,  
and $z^{k-1}G\in\mathcal{HC}(\mu)$.  

In this article, we are mainly interested in indivisible intervals inside the chain.

The reproducing kernel $K_{t_n}$ of $\mathcal{HC}(\mu)$ at $t_n\in T$, 
$$
\langle F,K_{t_n}\rangle=F(t_n),\qquad  F\in \mathcal{HC}(\mu),
$$
is given by
$$
K_{t_n}(z)=\mu_nA'(t_n)\frac{A(z)}{z-t_n},\qquad n\in\mathcal{N}.
$$ 
Therefore,  
$$
\|K_{t_n}\|=\mu^{1\slash2}_n|A'(t_n)|,\qquad n\in\mathcal{N},
$$
and 
$$
\frac{K_{t_n}(z)}{\|K_{t_n}\|}=\frac{
\|K_{t_n}\|}{A'(t_n)}\frac{A(z)}{z-t_n},\qquad n\in\mathcal{N}.
$$

\begin{lemma}\label{le1} If the chain ${\rm Chain}(\mu)$ contains an indivisible   
interval, and $A$ is associated to $\mathcal{HC}(\mu)$, then there exist entire functions $S$ and $G$ real on the real line such that $G\in \mathcal{HC}(\mu)$, $S\in\ell^2(\mu)$, and 
$$
\frac{GS}{A}=\sum_{n\in\mathcal N}\frac{a^2_n}{\cdot-t_n},
$$
where $a_n=\mu^{-1\slash2}_nG(t_n)/A'(t_n)$.

In the opposite direction, if there exist two entire functions 
$S$ and $G$ real on the real line such that $\lim_{|y|\rightarrow\infty}y^{k-1}G(iy)\slash A(iy)=0$ and
\begin{align}S&\in\ell^2(\mu), \label{er1}\\
\frac{G}{A'}&\in\ell^2(1\slash\mu), \label{er2}\\
\frac{GS}{A}&=\sum_{n\in\mathcal N}\frac{c_n}{\cdot-t_n}\label{er3},
\end{align}
where $\{c_n\}_{n\in\mathcal N}\in\ell^1$, $\sum_{n\in\mathcal N} c_n\ne 0$, then 
the chain ${\rm Chain}(\mu)$ contains an indivisible  
interval corresponding to a subspace of exponential type coinciding with that of $G$.
\end{lemma}

\begin{proof}
Let $G\in \mathcal{HC}(A,\mu)$ be an entire function real on the real line with simple real zeros, orthogonal to $G\slash(\cdot-\lambda)$, $\lambda\in\mathcal{Z}_G$. Then we have  
$$
G=A\sum_{n\in\mathcal N}\frac{a_n\mu^{1\slash2}_n}{\cdot-t_n}=\sum_{n\in\mathcal N}a_n\frac{K_{t_n}}{\|K_{t_n}\|}\cdot \frac{A'(t_n)}{|A'(t_n)|},
$$
with real coefficients $a_n$.
Then 
\begin{gather}
G\perp\frac{G}{\lambda-\cdot},\qquad \lambda\in\mathcal{Z}_G \qquad \Leftrightarrow\nonumber \\
\Bigl\langle \frac{G}{\lambda-\cdot}, \sum_{n\in\mathcal N} a_n\frac{K_{t_n}(z)}{\|K_{t_n}\|}\cdot \frac{A'(t_n)}{|A'(t_n)|}\Bigr\rangle=0,\qquad \lambda\in\mathcal{Z}_G
\quad\Leftrightarrow\nonumber \\
\sum_{n\in\mathcal N}\frac{a^2_n}{\lambda-t_n}=0,\qquad \lambda\in\mathcal{Z}_G,
\label{eq1}
\end{gather}
because
\begin{equation}
G(t_n)=a_nA'(t_n)\mu^{1\slash2}_n.
\label{eq2}
\end{equation}

Next, \eqref{eq1} is equivalent to the existence of an entire function $S$ real on the real line
such that
\begin{equation}
\sum_{n\in\mathcal N}\frac{a^2_n}{\cdot-t_n}=\frac{GS}{A}.
\label{eq5}
\end{equation}

Comparing the residues on $T$, we obtain that 
$$
a_n^2=\frac{G(t_n)S(t_n)}{A'(t_n)},\qquad n\in\mathcal N,
$$
and hence, 
\begin{equation}
S(t_n)=a_n\mu^{-1\slash2}_n,\qquad n\in\mathcal N.
\label{eq3}
\end{equation}

Finally, \eqref{eq2}--\eqref{eq3} yield \eqref{er1}--\eqref{er3} with $c_n\ge 0$, $0<\sum_{n\in\mathcal N}c_n<\infty$.

In the opposite direction, suppose that we can find two entire function 
$S$ and $G$ real on the real line such that $\lim_{|y|\rightarrow\infty}G(iy)\slash A(iy)=0$, and relations \eqref{er1}--\eqref{er3} hold 
with $\{c_n\}_{n\in\mathcal N}\in\ell^1$, $\sum_{n\in\mathcal N} c_n\ne 0$. 

Set $b_n=S(t_n)\mu^{1\slash2}_n$, $n\in\mathcal{N}$, and consider
$$
H= \sum_{n\in\mathcal N}b_n\frac{K_{t_n}}{\|K_{t_n}\|}\cdot \frac{A'(t_n)}{|A'(t_n)|} \in\mathcal{HC}(A,\mu). 
$$
Since $G\slash A'\in\ell^2(1\slash\mu)$ and $\lim_{|y|\rightarrow\infty}G(iy)\slash A(iy)=0$, a result from the de Branges theory \cite[Theorem 26]{bo1} 
yields that $G\in\mathcal{HC}(A,\mu)$. 

Set
$$
a_n=\frac{G(t_n)}{A'(t_n)\mu^{1\slash2}_n},\qquad n\in\mathcal N. 
$$
Then $a_nb_n=c_n$, $n\in\mathcal N$, 
$$
\langle G, H\rangle= \sum_{n\in\mathcal N}a_nb_n=\sum_{n\in\mathcal N}c_n\ne 0,
$$ 
and
\begin{multline*}
\Bigl\langle\frac{G}{\lambda-\cdot}, H\Bigr\rangle=\sum_{n\in\mathcal N}\frac{G(t_n)b_n}{(\lambda-t_n)\|K_{t_n}\|}\cdot \frac{A'(t_n)}{|A'(t_n)|}\\ =
\sum_{n\in\mathcal N}\frac{a_nb_n}{\lambda-t_n}=\frac{GS}{A}(\lambda)=0,\qquad \lambda\in\mathcal{Z}_G.
\end{multline*}
Thus $G\not\in\Span\{G\slash(\cdot-\lambda)\}_{\lambda\in\mathcal{Z}_G}$, and, by Lemma~\ref{le7}, we get an indivisible interval.
\end{proof}

If $S$ and $G$ in the formulation of Lemma~\ref{le1} are not polynomials, then the indivisible interval we obtain is inside the chain.

\begin{lemma}\label{le2} If the chain ${\rm Chain}(\mu)$ contains 
$k$ contiguous indivisible intervals, and $A$ is associated to $\mathcal{HC}(\mu)$, then there exist entire functions $S$ and $G$ real on the real line such that
$z^{k-1}G\in \mathcal{HC}(A,\mu)$, 
\begin{align*}S&\in\ell^2(\mu), \\
\frac{GS}{A}&=\sum_{n\in\mathcal N}\frac{a^2_n}{\cdot-t_n},
\end{align*}
where $a_n=\mu^{-1\slash2}_nG(t_n)/A'(t_n)$.

In the opposite direction, if there exist two entire functions 
$S$ and $G$ real on the real line such that $\lim_{|y|\rightarrow\infty}y^{k-1}G(iy)\slash A(iy)=0$ and
\begin{align*}
S&\in\ell^2(\mu), \\
\frac{z^{k-1}G}{A'}&\in\ell^2(1\slash\mu), \\
\frac{GS}{A}&=\sum_{n\in\mathcal N}\frac{c_n}{\cdot-t_n},
\end{align*}
where $\{c_n\}_{n\in\mathcal N}\in\ell^1$, $\sum_{n\in\mathcal N} c_n\neq0$, then 
the chain ${\rm Chain}(\mu)$ contains $k$ contiguous indivisible intervals.
\end{lemma}

The proof is analogous to that of Lemma~\ref{le1}. 

Again, if $S$ and $G$ are not polynomials, then the contiguous indivisible intervals we obtain are inside the chain.

\section{The spectrum $\mathbb{Z}$ and its perturbations}
\label{se3}

Here, we start with the case when the spectrum $T$ of the de Branges space is $\mathbb Z$, and, correspondingly, $A(z)=A_0(z)=\sin(\pi z)$.

\subsection{Indivisible interval inside the chain; proof of Theorem~\ref{onetwoint}}


\begin{proof}[Proof of Theorem~\ref{onetwoint}]
By Lemma~\ref{le2} we know that the existence of $k$ contiguous intervals is equivalent to the existence of two non-zero entire functions $S$ and $G$
real on the real line and such that $z^{k-1}G\in\mathcal{HC}(A_0,\mu)$,
\begin{equation}
\begin{cases}S\in\ell^2(\mu), \\
z^{k-1}G\in\ell^2(1\slash\mu), \\
\dfrac{GS}{A_0}=\sum\limits_{n\in\mathbb Z}\dfrac{c_n}{\cdot-n},
\end{cases}
\label{eq7}
\end{equation}
$\{c_n\}_{n\in\mathbb Z}\in\ell^1$. Additionally, we could impose the restriction $c_n\ge 0$, $n\in\mathbb{Z}$.

Now for $k=1$, choose entire functions $G$ and $S$ real on the real line such that $G(z)S(z)=z^{-1}A_0(z)$,
$|G(x)|\asymp\dist(x,\mathcal{Z}_G)(1+|x|)^{-1\slash2}$, $|S(x)|\asymp\dist(x,\mathcal{Z}_S)(1+|x|)^{-1\slash2}$. 
(For example, we can take $G(z)=\prod_{n\geq1}\bigl{(}1-\frac{z}{2n-1}\bigr{)}\bigl{(}1+\frac{z}{2n}\bigr{)}$). Then $\lim_{|y|\to\infty}G(iy)/A_0(iy)=0$. 

Set 
$$
\mu_n=\begin{cases}|n|^{-1\slash2},\quad n\in\mathcal{Z}_G, \\
|n|^{1\slash2}, \quad n\in\mathcal{Z}_S, 
\end{cases}
$$
and $\mu_0=1$. The measure $\mu=\sum_{n\in\mathbb Z}\mu_n\delta_n$ satisfies conditions \eqref{mu1} and \eqref{mu2}. 
Furthermore, conditions \eqref{eq7} are satisfied and the space $\mathcal{HC}(\mu)$ contains an indivisible interval inside the chain.

In the opposite direction, suppose that there are two contiguous indivisible intervals. Conditions \eqref{mu1}, \eqref{mu2}
and \eqref{eq7} imply that
\begin{align*}
&\sum_{n\in\mathbb Z}(S^2(n)+(1+|n|)^{-2})\mu_n<\infty,\\
&\sum_{n\in\mathbb Z}(n^2G^2(n)+(1+|n|)^{-2})\mu^{-1}_n<\infty.
\end{align*}

Therefore,
$$
\sum_{n\in\mathbb Z}(|S(n)|+(1+|n|)^{-1})(n|G(n)|+(1+|n|)^{-1})<\infty.
$$
By the Cartwright theorem \cite{bbook}, $S$ and $G$ have strictly positive exponential types. Since $t(S)+t(G)\le \pi$, these exponential types are smaller than $\pi$.

Since $G\in\ell^1(\mathbb{Z})$, we conclude that $G\in L^1(\mathbb{R})$ \cite[Section 10.6]{bbook}. In a similar way, since $S(n)(1+|n|)^{-1}\in\ell^1(\mathbb{Z})$, we have $S(x)(1+|x|)^{-1}\in L^1(\mathbb{R})$ \cite[Theorem~3a]{Agmon}.
Next we use that
$$
\frac{GS}{A_0}=\sum_{n\in\mathbb{Z}}\frac{a^2_n}{\cdot-n},\qquad c_n\ge 0,\, n\in\mathbb{Z},\qquad \{c_n\}_{n\in\mathbb{Z}}\in\ell^1.
$$
Using a version of Boole's lemma by Khrushchev--Vinogradov \cite{KV}, we obtain that 
$|G(x)S(x)|\asymp |x|^{-1}$ on a set $E\subset\mathbb{R}$ of infinite logarithmic length.
As a result,
\begin{multline*}
\infty=\biggl{(}\int_E\frac{dt}{1+|t|}\biggr{)}^2\leq \int_E|G(t)|\,dt\int_E\frac{|S(t)|\,dt}{1+|t|}\\ \le
\int_{\mathbb{R}}|G(t)|\,dt\int_{\mathbb{R}}\frac{|S(t)|\,dt}{1+|t|}<\infty.
\end{multline*}
This contradiction shows that no de Branges space with spectrum $\mathbb{Z}$ possesses
two contiguous indivisible intervals.
\end{proof}


Next, we consider some situations where the spectrum of our de Branges space is a perturbation of $\mathbb Z$.

\begin{proposition}
Suppose that for some $\gamma\in\mathbb{R}$ and $T\subset\mathbb R$ we have
$$
|A(z)|\asymp\min(1,\dist(z, T))(1+|z|)^{\gamma}\exp(\pi |\Im z|),\qquad z \in\mathbb C.
$$
If a de Branges space $\mathcal{HC}(A,\mu)$ is regular, then it can contain
$k$ contiguous indivisible intervals if and only if $k<2+\gamma$. 
\end{proposition}

In the proof, we use a possibility to factorize such entire functions $A$ into factors of precise asymptotics. For a similar arguments, see 
Lemmas~\ref{dfs1} and \ref{dfs2} below.

\begin{proof} 
Let $1\le k<2+\gamma$. Choose $\alpha\in(k-\gamma-1,1)$. Then choose $\beta\in(-1-2\gamma,\min(1,1-2\alpha-2\gamma))$, $\delta\in(\max(-1-2\gamma,2k-2\alpha-2\gamma-1),1)$ 
and define entire functions $G$ and $S$ real on the real line such that $GS=A/(\cdot-\lambda)$ for some $\lambda\in T$,
$|G(t)|\asymp \dist(t,\mathcal{Z}_G)(1+|t|)^{-\alpha}$, $|S(t)|\asymp \dist(t,\mathcal{Z}_S)(1+|t|)^{\gamma-1+\alpha}$, $t\in\mathbb R$, and a measure 
$\mu=\sum_{t\in T}\mu_t\delta_t$ with $\mu_t=(1+|t|)^{\beta}$, $t\in \mathcal{Z}_G$, $\mu_t=(1+|t|)^{\delta}$, $t\in \mathcal{Z}_S$, $\mu_\lambda=1$.  
It remains to apply Lemma~\ref{le2}. 
In the opposite direction, we argue by analogy with the proof of Theorem~\ref{onetwoint}.
\end{proof}

\begin{proposition} Given $\beta>0$, set $T=\cup_{n\in\mathbb Z}\{n,n+(2+|n|)^{-\beta}\}$.
If the space $\mathcal{HC}(T,\mu)$ is regular, then $\beta<1$, and the corresponding 
de Branges subspaces chain ${\rm Chain}(\mu)$ can contain an indivisible interval 
and cannot contain two contiguous indivisible intervals. 
\end{proposition}

\begin{proof}
Set
$$
A(z)=\sin(\pi z)\cdot\prod_{n\in\mathbb{Z}}\biggl{(}1-\frac{z}{n+(2+|n|)^{-\beta}}\biggr{)}.
$$
Then 
$$
|A'(t)|\asymp (1+|t|)^{-\beta},\qquad t\in T.
$$
Since the space $\mathcal{HC}(T,\mu)$ is regular, by \eqref{mu1} and \eqref{mu2} we obtain that
$$
\sum_{t\in  T} (1+|t|)^{\beta-2}<\infty,
$$
and, hence, $\beta<1$.

Next, if ${\rm Chain}(\mu)$ contains two contiguous indivisible intervals, then, as in the proof of Theorem~\ref{onetwoint}, relations \eqref{mu1}, \eqref{mu2}, 
and \eqref{eq7} imply that
\begin{align*}
\sum_{t\in T}(S^2(t)+(1+|t|)^{-2})\mu_t&<\infty,\\
\sum_{t\in T}(t^{2\beta+2}G^2(t)+(1+|t|)^{2\beta-2})\mu^{-1}_t&<\infty,
\end{align*}
and we conclude as in the proof of Theorem~\ref{onetwoint}.

If now $\beta\in(0,1)$, let us verify that ${\rm Chain}(\mu)$ can contain an indivisible interval. Given $n\in\mathbb Z$, set $n^*=n+(2+|n|)^{-\beta}$. Choose $\Lambda\subset\mathbb Z$ such that $0\not\in \Lambda$ and 
$$
S(x)=\prod_{n\in\Lambda}\Bigl(1-\frac{x}n\Bigr)\asymp (1+|x|)^{-(\beta+1)/2},\qquad x\in\mathbb Z\setminus\Lambda,
$$
and set $G(z)=A(z)/(zS(z))$. 
Denote $\Lambda^*=\{n^*:n\in\Lambda\}$. 
We have
\begin{align*}
|S(t)|&\asymp
\begin{cases}
(1+|t|)^{-(3\beta+1)/2},\qquad t\in \Lambda^*,\\
(1+|t|)^{-(\beta+1)/2},\qquad t\in T\setminus(\Lambda\cup\Lambda^*),
\end{cases}\\
|G(t)|&\asymp (1+|t|)^{-(\beta+1)/2},\qquad t\in \Lambda\cup\{0\}.
\end{align*}
Now, we set
$$
\mu_t=
\begin{cases}
(1+|t|)^{(\beta+1)/2},\qquad t\in \Lambda\cup\{0\},\\
(1+|t|)^{(3\beta-1)/2},\qquad t\in T\setminus(\Lambda\cup\{0\}),
\end{cases}
$$
A direct calculation shows that the measure $\mu=\sum_{t\in T}\mu_t\delta_t$, satisfies conditions \eqref{mu1} and \eqref{mu2}. Furthermore, 
$\lim_{|y|\to\infty}G(iy)/A(iy)=0$,  
$S\in\ell^2(\mu)$, $G\slash A'\in\ell^2(1\slash\mu)$, and we conclude by applying Lemma~\ref{le1}.
\end{proof}

\subsection{Infinite number of indivisible intervals. Proof of Theorem~\ref{th2}}
\leavevmode

In the following result we consider lacunary canonical products constructed by rapidly growing zeros $\{z_k\}_{k\ge 1}$, $|z_{k+1}/z_k|\ge q>1$, 
$k\ge 1$.

\begin{proposition}
Let $U$ be a lacunary canonical product $\Lambda\subset\mathbb{R}$, $\dist(\Lambda,\mathbb{Z})>0$, $T=\Lambda\cup \mathbb Z$, 
and let $A(z)=\sin(\pi z)U(z)$. Then we can find a measure $\mu$ on $T$ such that the corresponding space $\mathcal{HC}(A,T)$ contains a two sided sequence of infinitely
many contiguous indivisible intervals. 
\end{proposition}

\begin{proof} Here we just choose entire functions $G$ and $S$ real on the real line such that $zG(z)S(z)=A(z)$, $|G(t)|\asymp \dist(t,\mathcal{Z}_G)\psi(t)
(1+|t|)^{-1/2}$, $|S(t)|\asymp \dist(t,\mathcal{Z}_S)\psi(t)
(1+|t|)^{-1/2}$, where $\psi(t)=1+\max_{|z|=t}|U(z)|^{1/2}$, $t\in\mathbb R$. Then set $\mu_t=\psi(t)^{-3}$, $t\in\mathcal{Z}_G$, $\mu_t=1$, $t\in\mathcal{Z}_A\setminus\mathcal{Z}_G$. Then we obtain that $z^kS\in\ell^2(\mu)$, $z^kG/A'\in\ell^2(1/\mu)$,  for any $k\ge 0$ and apply a natural analog of Lemma~\ref{le2}.  
\end{proof}

Next we need some standard information on the asymptotics of canonical products associated with very regular sequences on the real line.  
 
Given a countable symmetric $\Lambda\subset\mathbb R\setminus\{0\}$ of finite linear density, we set 
$$
\mathcal C_\Lambda(z)=\prod_{t\in\Lambda}\Bigl(1-\frac{z}{t}\Bigr)=\prod_{t\in\Lambda_+}\Bigl(1-\frac{z^2}{t^2}\Bigr),
$$
where $\Lambda_+=\Lambda\cap\mathbb R_+$. 
Denote by $n_\Lambda$ the counting function of $\Lambda$,
$$
n_\Lambda(x)=\card\bigl(\Lambda\cap(0,x]\bigr),\qquad x>0.
$$
We say that a symmetric $\Lambda\subset\mathbb R\setminus\{0\}$ has strong linear asymptotics $a_\Lambda x+b_\Lambda$ if the function 
$$
\psi_\Lambda(t)=\int_0^t \bigl( n_\Lambda(x)-\lfloor a_\Lambda x+b_\Lambda\rfloor\cdot\mathbf 1_{[1,\infty)}(x) \bigr)\,dx
$$
is bounded on $(0,\infty)$. We say that $\Lambda\subset\mathbb R$ is uniformly discrete if $\inf\bigl\{|t_1-t_2|:t_1,t_2\in\Lambda,\, t_1\not= t_2\bigr\}>0$. 

\begin{lemma} Given an infinite symmetric uniformly discrete $\Lambda\subset\mathbb R\setminus\{0\}$ with strong linear asymptotics $a x+b$, we have 
$$
|\mathcal C_\Lambda(z)|\asymp \min(1,\dist(z,\Lambda))(1+|z|)^{-1-2b}e^{a\pi|\Im z|},\qquad z\in\mathbb C.
$$
\label{dfs1}
\end{lemma}

\begin{proof}
Since the linear density of $\lambda$ is $a$, we need only to verify that 
$$
|\mathcal C_\Lambda(x+i)|\asymp (1+|x|)^{-1-2b}, \qquad x\in\mathbb R.
$$
It is easily seen that for $\Lambda_{a,b}=\{\pm(n-b)/a\}_{n\ge 1}$ (with trivial modifications for small $n$) we have $n_{\Lambda_{a,b}}(x)=\lfloor ax+b\rfloor\cdot\mathbf 1_{[1,\infty)}(x)$,
$$
|\mathcal C_{\Lambda_{a,b}}(x+i)|\asymp (1+|x|)^{-1-2b},\qquad x\in\mathbb R.
$$
Therefore, we need only to check that the function $W$,
$$
W(x)=\log\Bigl| \frac{\mathcal C_\Lambda(x+i)}{\mathcal C_{\Lambda_{a,b}}(x+i)} \Bigr|=\int_0^\infty \log\Bigl| 1-\Bigl(\frac{x+i}{t}\Bigr)^2 \Bigr|\,\bigl(dn_\Lambda(t)-dn_{\Lambda_{a,b}}(t)\bigr),
$$
is bounded on the real line. Integrating by parts twice and using that $n_\Lambda-n_{\Lambda_{a,b}}$ and $\psi_\Lambda$ are bounded, we obtain that 
$$
W(x)=\int_0^\infty \psi_\Lambda(t)\cdot \Re\, \Bigl[\frac{2}{t^2}-\frac1{(t-x+i)^2}-\frac1{(t+x+i)^2}\Bigr] \,dt.
$$
The function in the right hand side is bounded because $\psi_\Lambda$ are bounded.
\end{proof}

\begin{lemma} Let $\Lambda_1\subset\Lambda_2$ be two symmetric subsets of $\mathbb R\setminus\{0\}$ 
with strong linear asymptotics, correspondingly, $a_1 x+b_1$ and $a_2 x+b_2$. Given $a\in(a_1,a_2)$ and $b\in(b_1,b_2)$, there exists 
a symmetric subset $\Lambda$ of $\mathbb R\setminus\{0\}$ with linear asymptotics $ax+b$ such that $\Lambda_1\subset\Lambda\subset\Lambda_2$.
\label{dfs2}
\end{lemma}

\begin{proof} By observation.
\end{proof}

\begin{proof}[Proof of Theorem~\ref{th2}] We consider just the case of infinite $\Sigma$. The other case is much simpler. Let $\Sigma=\{\pi s_k\}_{k\ge 1}$. 
By induction in $k\ge 1$, we construct a disjoint system of intervals $(a_k,b_k)\subset (0,1)$ such that $s_k<s_m\implies b_k<a_m$, $k,m\ge 1$. Set $r_k=(a_k+b_k)/2$, $k\ge 1$. 
Also by induction in $k\ge 1$, we construct, using Lemma~\ref{dfs2}, symmetric sets $\Lambda_k\subset\mathbb Z\setminus\{0\}$ with 
strong linear asymptotics $s_k x-(1+r_k)/6$ such that $s_k<s_m\implies \Lambda_k\subset \Lambda_m$, $k,m\ge 1$. 

Set $G_k=\mathcal C_{\Lambda_k}$, $S_k=\sin(\pi z)/(zG_k(z))$, $k\ge 1$. For every $k\ge 1$, by Lemma~\ref{dfs1} we have 
\begin{align}
|G_k(n)|&\asymp (1+|n|)^{(r_k-2)/3}, \qquad n\in\mathbb Z\setminus \Lambda_k,\label{dfs3}\\
|S_k(n)|&\asymp (1+|n|)^{-(r_k+1)/3}, \qquad n\in\Lambda_k.\label{dfs4}
\end{align}

Given $n\in\mathbb Z\setminus\{0\}$, we set 
$$
s(n)=\sup\,\bigl\{s_k:k\ge1,\,n\not\in\Lambda_k\bigr\}.
$$
By construction, if $s_k>s(n)$, then $n\in\Lambda_k$, and if $s_k<s(n)$, then $n\not\in\Lambda_k$. 
If $s(n)=s_m\in\Sigma$ and $n\in\Lambda_m$, then we set $u(n)=a_m$, otherwise set 
$$
u(n)=\sup\,\bigl\{b_k:s_k<s(n)\bigr\}.
$$
Now, we set $\mu_0=1$, 
$$
\mu_n=(1+|n|)^{(2u(n)-1)/3}, \qquad n\in \mathbb Z\setminus\{0\}.
$$
Then the measure $\mu=\sum_{n\in\mathbb Z}\mu_n\delta_n$ satisfies conditions \eqref{mu1} and \eqref{mu2}.  

Fix $k\ge 1$. To prove the existence of an indivisible interval corresponding to the exponential type $s_k$, by Lemma~\ref{le1}, 
we need only to verify that 
$S_k\in\ell^2(\mu)$ and $G_k\in\ell^2(1/\mu)$.

If $n\in\Lambda_k$, then $s(n)\le s_k$ and $u(n)\le a_k$. Therefore, by \eqref{dfs4}, we have 
\begin{multline*}
\sum_{n\in\Lambda_k}|S_k(n)|^2\mu_n\asymp \sum_{n\in\Lambda_k}(1+|n|)^{-(2r_k+2)/3}(1+|n|)^{(2u(n)-1)/3}\\ \le 
\sum_{n\in\mathbb Z}(1+|n|)^{-(2r_k+2)/3+(2a_k-1)/3}=\sum_{n\in\mathbb Z}(1+|n|)^{-1+2(a_k-r_k)/3}<\infty.
\end{multline*}
If $n\in\mathbb Z\setminus(\Lambda_k\cup\{0\})$, then $s(n)\ge s_k$ and $u(n)\ge b_k$. Therefore, by \eqref{dfs3}, we have 
\begin{multline*}
\!\!\sum_{n\in\mathbb Z\setminus(\Lambda_k\cup\{0\})}|G_k(n)|^2\mu^{-1}_n\asymp \sum_{n\in\mathbb Z\setminus(\Lambda_k\cup\{0\})}(1+|n|)^{(2r_k-4)/3}(1+|n|)^{(1-2u(n))/3}\\ \le 
\sum_{n\in\mathbb Z}(1+|n|)^{(2r_k-4)/3+(1-2b_k)/3}=\sum_{n\in\mathbb Z}(1+|n|)^{-1+2(r_k-b_k)/3}<\infty.
\end{multline*}
This completes the proof.
\end{proof}

\section{The same type subspaces}
\label{se4}

\subsection{Regularity of growth of exponential type. Proof of Theorem 5}

The proof of Theorem~\ref{regth} is based on a combination of an atomization result in \cite{B24} and some fact on the completeness of mixed systems in the Paley--Wiener spaces from \cite{B25}.

\begin{proof}[Proof of Theorem~\ref{regth}]
We start with the following simple fact. If two positive weights are comparable, that is $w_1(x)\asymp w_2(x)$, 
then the chains of the de Branges subspaces are the same, i.e. the de Branges subspaces from different chains coincide as sets with equivalent norms. 
Therefore, it is sufficient to consider any weight comparable to $w$. 

Now we apply Theorem~2.6 from \cite{B24} (with sufficiently large $\sigma>0$) and
construct an entire function $H$ of finite exponential type $b$ with simple zeros such that 
$$
|H(x)|^2\asymp w(x),\qquad x\in\mathbb R.
$$
It remains to prove that the measure $|H(x)|^2\,dm$ generates a thin chain. 
Assume the contrary. Then there exist two different de Branges spaces $\mathcal{H}_1$, $\mathcal{H}_2$ from the chain of the same exponential type.
Let us fix some non-trivial function $F_1$ from $\mathcal{H}_1\ominus \mathcal{H}_2$.
Let $F_2$ be an $A$-function corresponding to $\mathcal{H}_2$ such that $\mathcal{Z}_{F_2}\cap \mathcal{Z}_G=\emptyset$.
Set $a=t(F_1)=t(F_2)$. We have 
$$
F_1\perp\frac{F_2(z)}{z-\lambda}, \qquad \lambda\in \mathcal{Z}_{F_2},
$$
where $\perp$ means orthogonality in $\mathcal{H}_1$. We recall that the space $\mathcal{H}_1$ is isometrically embedded
in $L^2(|H|^2\,dm)$.
Hence, 
\begin{equation}
\int_{\mathbb{R}}\frac{F_1(x)H(x)\overline{F_2(x)H(x)}\,dx}{x-\lambda}=0,\qquad \lambda\in\mathcal{Z}_{F_2}.
\label{ort1}
\end{equation}
Since $F_1H,F_2H/(\cdot-\lambda)\in L^2(\mathbb{R})$, $\lambda\in\mathcal{Z}_{F_2}$, the functions $F_1G,F_2H/(\cdot-\lambda)$, $\lambda\in\mathcal{Z}_{F_2}$, belong to the Paley--Wiener space $\mathcal{P}W_{a+b}$. 
Thus, equation \eqref{ort1} can be considered as orthogonality of some vectors from $\mathcal{P}W_{a+b}$. 
\smallskip
Denote by $k_\lambda$ the reproducing kernel in the space $\mathcal{P}W_{a+b}$ at the point $\lambda\in\mathbb C$.
From \eqref{ort1} we obtain that the system
$$
\bigl\{k_{\lambda}\bigr\}_{\lambda\in \mathcal{Z}_H}\cup\biggl{\{}\frac{F_2H}{\cdot-\lambda}\biggr{\}}_{\lambda\in\mathcal{Z}_{F_2}}
$$
is not complete in $\mathcal{P}W_{a+b}$. This contradicts the following lemma:

\begin{lemma}
Let $T=T_1T_2$ be an entire function in the Paley--Wiener space $\mathcal{PW}_\pi$ with the conjugate indicator
diagram $[-\pi i,\pi i]$ and with simple zeroes. Then
the mixed system
$$
\bigl\{k_{\lambda}\bigr\}_{\lambda\in \mathcal{Z}_{T_1}}\cup\biggl\{\frac{T_1T_2}{\cdot-\lambda}\biggr\}_{\lambda\in\mathcal{Z}_{T_2}}$$
is always complete in $\mathcal{PW}_\pi$.
\end{lemma}

This lemma follows immediately from \cite[Proposition 2.1]{B25}. For other versions of this result see \cite{B26,B27}.
\end{proof}

\subsection{Spectrum $\mathbb Z$. Proof of Theorem~\ref{thm6}}

The proof of Theorem \ref{thm6} is based on a combination of Theorem~\ref{onetwoint}, some results on the classical P\'olya problem, 
and the second Beurling--Malliavin theorem.

\begin{proof}[Proof of Theorem~\ref{thm6}] 
First of all, since $\mathcal H=\mathcal H_{L,\nu}=\mathcal H(\sin (\pi z),\mu)$ is regular and the support of $\nu$ and $\mu$ is $\mathbb Z$, 
by the Cartwright theorem \cite{bbook}, the space $\mathcal H$ contains no entire functions of zero exponential type except $0$. 

Set $\mathcal{H}_1=\mathcal H_{t_1,\nu}$, $\mathcal{H}_2=\mathcal H_{t_2,\nu}$.  Using Theorem~\ref{onetwoint} we obtain that if $\dim(\mathcal{H}_1\ominus\mathcal{H}_2)<\infty$, then
$\dim(\mathcal{H}_1\ominus\mathcal{H}_2)\leq1$. Thus, it remains to consider the case $\dim(\mathcal{H}_1\ominus\mathcal{H}_2)=\infty$, 
${\rm Type}(\mathcal{H}_1)={\rm Type}(\mathcal{H}_2)$.

Choose a function $F\in\mathcal{H}_1\setminus\{0\}$
such that $F\perp \mathcal{H}_2$. Let $A_2$ be an $A$-function corresponding to the space $\mathcal{H}_2$ such that $\mathcal{Z}_{A_2}\cap \mathbb Z=\emptyset$.
We have 
\begin{equation}
F\perp \frac{A_2}{\cdot-s_n},\qquad  s_n\in\mathcal{Z}_{A_2}.
\label{orth1}
\end{equation}
Now, relation \eqref{orth1} is equivalent to the interpolation formula
\begin{equation}
\sum_{n\in\mathbb{Z}}\frac{F(n)A_2(n)}{\mu_n(z-n)}=\frac{A_2(z)S(z)}{\sin(\pi z)},
\label{orth2}
\end{equation}
for some entire function $S$. Since there exists infinitely many linear independent functions $F$ satisfying \eqref{orth2}
we can assume that the functions $F$ and $S$ have at least $100$ common non-integer zeroes $\lambda_1,\ldots,\lambda_{100}$.
Set $P(z)=\prod_{k=1}^{100}(z-\lambda_k)$. From \eqref{orth2} we conclude that 
$$
\sum_{n\in\mathbb{Z}}\frac{F(n)A_2(n)}{(n-\lambda_1)\mu_n(z-n)}=\frac{A_2(z)S(z)}{(z-\lambda_1)\sin(\pi z)},
$$
and then, by induction, 
$$
\sum_{n\in\mathbb{Z}}\frac{F(n)A_2(n)}{P(n)\mu_n(z-n)}=\frac{A_2(z)S(z)}{P(z)\sin(\pi z)}.
$$ 
Hence,
\begin{equation}
\frac{F}{P}\perp \frac{A_2}{\cdot-s_n},\qquad  s_n\in\mathcal{Z}_{A_2}.
\label{orth4}
\end{equation}
Thus, we can assume that our function $F$ satisfies the inequality $|F(x)|\le |x|^{-100}$, and is real on the real line. 
Using Lemmas~\ref{ap3},\ref{ap1} we find such a representation $\mathcal{HC}(T,\gamma)$ of our space that 
the zeroes of $F$ on the real line are away from the support $(x_n)_{n\in\mathbb Z}$ of $\gamma=\sum_{n\in\mathbb Z}\gamma_n\delta_{x_n}$:  
\begin{equation}
\text{the set\ }\mathcal{Z}_F\cap \bigcup_{n\in\mathbb Z\setminus\{0\}}\bigl[x_n-|n|^{-10},x_n+|n|^{-10}\bigr]\text{\ is bounded}.
\label{dr1}
\end{equation}

By \eqref{orth4}, using this representation $\mathcal{HC}(T,\nu)$, we obtain that 
$$
\sum_{n\in\mathbb{Z}}\frac{F(x_n)A_2(x_n)}{\gamma_n(\cdot-x_n)}=\frac{A_2U}{T}
$$
for some entire function $U$. Moreover, $t(A_2)+t(U)\le t(T)=1$ and  $0<t(F)\le t(A_2)$. By comparing residues we obtain 
\begin{gather}
U(x_n)=F(x_n)\gamma_n^{-1}T'(x_n),\nonumber\\
|U(x_n)|\lesssim (1+|x_n|)^{-10},\qquad  n\in\mathbb Z.
\label{dr2}
\end{gather}
Hence,
$$
\frac{UF}{T}=\sum_{n\in\mathbb{Z}}\frac{U(x_n)F(x_n)}{T'(x_n)(\cdot-x_n)}+R=
\sum_{n\in\mathbb{Z}}\frac{F^2(n)}{\gamma_n(\cdot-x_n)}+R,
$$
for some entire function $R$ of zero exponential type which is real on the real line.

{\bf Case 1.} $R$ is a polynomial. Then the zeroes of the product $UF$ are sufficiently close to the support of $\gamma$, which contradicts to \eqref{dr1}.

{\bf Case 2.} $R$ is a transcendental entire function of zero exponential type. 
The product $UF$ has at least one zero on every interval $(n,n+1)$. By \eqref{dr1} and \eqref{dr2}, $R$ is bounded on $\Sigma=\mathcal{Z}_F\cap\mathbb R$.

As in \cite{B26}, we use now some information on the classical P\'olya problem
and the second Beurling--Malliavin theorem. 

A sequence $X=\{x_n\} \subset \mathbb{R}$ is a P\'olya sequence 
if any entire function of zero exponential type which is bounded on $X$
is a constant. We say that a disjoint sequence of intervals 
$\{I_n\}$ on the real line is a long sequence of intervals if 
$$
\sum_n\frac{|I_n|^2}{1+\dist^2(0,I_n)}=+\infty.
$$

Since $\Sigma$ is not a P\'olya sequence and is a union of two separated
sequences, a theorem by Mitkovski--Poltoratski \cite{mp} (see also the discussion in \cite{B26}) gives that  
there exists a long sequence of intervals $\{I_n\}$ such that 
$$
\frac{\card(\Sigma \cap I_n)}{|I_n|}\to 0.
$$

Therefore, if $\Sigma'=\mathcal{Z}_U\cap\mathbb R$, then 
$$
\frac{\card(\Sigma' \cap I_n)}{|I_n|}\to 1.
$$

By the second Beurling--Malliavin theorem \cite{bm}, we obtain that $t(U)\ge 1$, and, hence, $t(F)$=0.
This contradiction completes the proof.
\end{proof}

\section{Isometric Cauchy transform representations for de Branges spaces}
\label{se5}

We start with two standard results. For reader's convenience we formulate them here and give the proofs.

\begin{lemma} 
\label{ap2}
Given a de Branges space $\mathcal{HC}(A,\mu)$, its reproducing kernel is 
$$
K_w(z)=A(z)\overline{A(w)}\sum_{n\in\mathcal N}\frac{\mu_n}{(z-t_n)(\overline{w}-t_n)}.
$$
If $w_1,w_2\not\in \supp\mu$, then
\begin{equation}
K_{w_1}(w_2)=A(w_2)\overline{A(w_1)}\frac{\psi(w_2)-\overline{\psi(w_1)}}{\overline{w_1}-w_2},
\label{drg1}
\end{equation}
where
$$
\psi(z)=\sum_{n\in\mathcal N}\mu_n\Bigl(\frac1{z-t_n}+\frac{t_n}{t^2_n+1}\Bigr).
$$
\end{lemma}

\begin{proof} By observation.
\end{proof}

\begin{lemma} 
\label{ap3}
Given a regular de Branges space $\mathcal{H}=\mathcal{HC}(A_0,\mu)$, $A_0(z)=\sin(\pi z)$, and $u\in\mathbb R$, set 
$$
T=A_0\cdot(\psi-u).
$$
Then $T$ is an entire function of exponential type real on the real line, with the conjugate indicator
diagram $[-\pi i,\pi i]$. 
For every $n\in\mathbb Z$, $T$ has exactly one simple zero $x_n$ on $(n,n+1)$, $\psi(x_n)=u$, and $\mathcal{Z}_T=\{x_n\}_{n\in\mathbb Z}$.

Next, $T\not\in \mathcal{H}$, $\{K_{x_n}\}_{n\in\mathbb Z}$ is an orthogonal basis in $\mathcal{H}$, and 
$$
\mathcal{H}=\mathcal{HC}(T,\nu),
$$
where
$$
\nu=\sum_{n\in\mathbb Z}\|T/(\cdot-x_n)\|^{-2}_{\mathcal{H}}\delta_{x_n}.
$$
\end{lemma}

\begin{proof} Since the zeros of $T$ and $A_0$ are interlacing, $T$ is of exponential type with the conjugate indicator
diagram $[-\pi i,\pi i]$, and 
\begin{equation}
|A_0(iy)|=O(|yT(iy)|),\qquad |y|\to\infty.
\label{drg}
\end{equation}

If $T=A_0\cdot(\psi-u)\in\mathcal{H}$, then
$$
A_0(z)\bigl(\psi(z)-u\bigr)=A_0(z)\sum_{n\in\mathbb Z}\frac{a_n\mu^{1\slash2}_n}{z-n}
$$
for some sequence $\{a_n\}_{n\in\mathbb Z}\in\ell^2$. 
Comparing the values at the integer points we obtain that $a_n=\mu^{1\slash2}_n$, $n\in\mathbb Z$, and, hence, $\sum_{n\in\mathbb Z} \mu_n<\infty$, which contradicts to 
\eqref{mu2}.

By formula \eqref{drg1}, $\{K_{x_n}\}_{n\in\mathbb Z}$ is an orthogonal system in $\mathcal{H}$. 
If $F\in\mathcal{H}\setminus\{0\}$ is orthogonal to $\{K_{x_n}\}_{n\in\mathbb Z}$, then $F=TS$ for some entire function $S$. By \eqref{drg}, we obtain that 
$|S(iy)|=O(|y|)$, $|y|\to\infty$. Furthermore, 
$$
S=\frac{F}{A_0}\cdot\frac{A_0}{T}
$$
is of zero exponential type. Therefore, $S$ is a polynomial of order at most $1$. Dividing $F$ by $S$ we obtain that $T\in \mathcal{H}$, which is impossible. 
Thus,  $\{K_{x_n}\}_{n\in\mathbb Z}$ is an orthogonal basis in $\mathcal{H}$. 

Denote 
$$
\nu_n=\|T/(\cdot-x_n)\|^{-2}_{\mathcal{H}}.
$$
Since $\{T/(\cdot-x_n)\}_{n\in\mathbb Z}$ is a biorthogonal system to $\{K_{x_n}\}_{n\in\mathbb Z}$, it is an orthogonal basis in $\mathcal{H}$. 
Hence, for every $f\in \mathcal{H}$ we have 
$$
f(z)=T(z)\sum_{n\in\mathbb Z}\frac1{z-x_n}\cdot\frac{\langle f,T/(\cdot-x_n) \rangle}{\|T/(\cdot-x_n)\|^2_{\mathcal{H}}}=
T(z)\sum_{n\in\mathbb Z}\frac{a_n\nu_n^{1/2}}{z-x_n},
$$
where
$a_n=\langle f,T/(\cdot-x_n) \rangle\nu_n^{1/2}$, $n\in\mathbb Z$, and $\{a_n\}_{n\in\mathbb Z}\in\ell^2$.
Therefore,
$$
\mathcal{H}\subset \mathcal{HC}(T,\nu),
$$ 
where $\nu=\sum_{n\in\mathbb Z}\nu_n\delta_{x_n}$, and the inclusion is isometric. Finally, again since $\{T/(\cdot-x_n)\}_{n\in\mathbb Z}$ is an orthogonal basis in $\mathcal{H}$, we have $\mathcal{H}=\mathcal{HC}(T,\nu)$. 
\end{proof}

Next we show that for every subset $\Lambda$ of $\mathbb R$ of finite upper linear density, we can find an isometric representation of our space with respect to a measure somewhat 
separated from $\Lambda$.

\begin{lemma} 
\label{ap1}
In the conditions of Lemma~\ref{ap3}, given a sequence of points $\{y_k\}_{k\ge 1}$ of finite upper linear density, we can find $u\in\mathbb R$ such that 
the intersection 
$$
\{y_k\}_{k\ge 1}\cap \bigcup_{n\in\mathbb Z\setminus\{0\}}\bigl[x_n-|n|^{-10},x_n+|n|^{-10}\bigr]
$$
is bounded.
\end{lemma}

\begin{proof} Set $h=\arctan\psi$. We have
$$
|h'(t)|=\frac{\sum_{n\in\mathbb Z}\frac{\mu_n}{(t-n)^2}}{\bigl(\sum_{n\in\mathbb Z}\mu_n\bigl(\frac1{t-n}+\frac{n}{n^2+1}\bigr)\bigr)^2+1}.
$$
Let $t\in(m,m+1)$. Without loss of regularity we can assume that $m\ge1$, $s=t-m\le 1/2$.
Since
$$
\sum_{n\in\mathbb Z}\frac{\mu_n}{n^2+1}<\infty,
$$
we have
\begin{align*}
\sum_{n\in\mathbb Z}\frac{\mu_n}{(t-n)^2}&\lesssim m^2+\mu_ms^{-2},\\
\sum_{n\in\mathbb Z}\mu_n\Bigl(\frac1{t-n}+\frac{n}{n^2+1}\Bigr)&\ge \mu_ms^{-1}-O(m^2).
\end{align*}
Hence, since 
$$
\sum_{n\in\mathbb Z}\frac1{(n^2+1)\mu_n}<\infty, 
$$
we obtain  
\begin{equation}
|h'(t)|=O(t^6),\qquad |t|\to\infty.
\label{drg3}
\end{equation}

Given $k\ge 1$, choose $n$ such that $y_k\in[n,n+1)$ and denote by $\ell_k$ the length of the set 
$$
J_k=h\bigl([y_k-2|n|^{-10},y_k+2|n|^{-10}]\bigr).
$$
By \eqref{drg3}, $\sum_{k\ge 1}\ell_k<\infty$.  Therefore, we can find $u\in\mathbb R$ which belongs to at most finitely many sets $J_k$. 
Then for sufficiently large $k$,
$$
\tan u\not\in \psi\bigl([y_k-2|n|^{-10},y_k+2|n|^{-10}]\bigr).
$$
If $\psi(s)=\tan u$, $s\in(n,n+1)$, then $s\not\in \bigl[y_k-2|n|^{-10},y_k+2|n|^{-10}\bigr]$ and, hence, $y_k\not\in \bigl[s-|n|^{-10},s+|n|^{-10}\bigr]$. 
The same is true for $s\in(n-1,n)\cup(n+1,n+2)$. This completes the proof.
\end{proof}


\begin{thebibliography}{BA}

\bibitem{Agmon} S.Agmon, Functions of exponential type in an angle and singularities of Taylor series, Trans.\ Amer.\ Math.\ Soc.\ {\bf 70} (1951) 492--508.

\bibitem{B25} A.Aleman, A.Baranov, Yu.Belov, Subspaces of $C^\infty$ invariant under the differentiation, J. Funct.\ Anal.\ {\bf 268} (2015) 2421--2439.

\bibitem{B26} A.Baranov, Yu.Belov, A.Borichev, A restricted shift completeness problem, J. Funct.\ Anal.\ {\bf 263} (2012) 1887--1893.

\bibitem{B27} A.Baranov, Yu.Belov, A.Borichev, Hereditary completeness for systems of exponentials and reproducing kernels, Adv.\ Math.\ {\bf 235} (2013) 525--554.

\bibitem{B24} Yu.Belov, Model functions with nearly prescribed modulus, St.\ Petersburg Math.\ J. {\bf 20} (2009) 2, 163--174.

\bibitem{re1} R.Bessonov, S.Denisov, De Branges canonical systems with finite logarithmic integral, Anal.\ PDE {\bf 14} (2021) 1509--1556.

\bibitem{bm} A.Beurling, P.Malliavin, On the closure of characters and the zeros of entire functions, Acta Math.\ {\bf 118} (1967) 79--93.

\bibitem{bbook} R.P.Boas, Entire functions, Academic Press, 1954. 

\bibitem{BS} A.Borichev, M.Sodin, Weighted exponential approximation and non-classical orthogonal spectral measures, Adv.\ Math.\ {\bf 226} (2011) 2503--2545.

\bibitem{bo1} L.de Branges, Hilbert spaces of entire functions, Prentice-Hall, 1968.

\bibitem{KV}  S.Khrushchev, S.Vinogradov, Free interpolation in the space of uniformly convergent Taylor series, Complex analysis and spectral theory (Leningrad, 1979/1980), 
pp.\ 171--213, Lecture Notes in Math.\ 864, 1981.

\bibitem{mp} M.Mitkovski, A.Poltoratski, P\'olya sequences, Toeplitz kernels and gap theorems, Adv.\ Math.\ {\bf 224} (2010) 1057--1070.

\bibitem{Po} A.Poltoratski, A problem on completeness of exponentials, Ann.\ of Math.\ (2) {\bf 178} (2013) 983--1016.

\bibitem{bo2} C.Remling, Spectral theory of canonical systems, De Gruyter Studies in Mathematics, 2018. 

\bibitem{bo3} R.Romanov, Canonical systems and de Branges spaces, arXiv:1408.6022.


\bibitem{re3} P.Yuditskii, Direct Cauchy theorem and Fourier integral in Widom domains, J. Anal.\ Math.\ {\bf 141} (2020) 411--439.

\end{thebibliography}
\end{document}